\documentclass{amsart}

\ifx\MyBeamer\undefined
\usepackage[breaklinks,colorlinks]{hyperref}
\usepackage[a4paper,margin=1in]{geometry}
\else
\if\MyBeamer t
\usepackage[breaklinks,colorlinks]{hyperref}
\usepackage[a4paper,margin=1in]{geometry}
\fi\fi

\usepackage{amsmath}
\usepackage{amssymb}
\usepackage{mathrsfs}
\usepackage{amsthm}
\usepackage{aliascnt}
\usepackage[utf8]{inputenc}
\usepackage[T1]{fontenc}
\usepackage[nofancy]{svninfo}
\ifx\FrenchText\undefined
\usepackage[british]{babel}
\else
\usepackage[british,french]{babel}
\AtBeginDocument{%
\catcode`\:=12%
\catcode`\!=12%
}
\fi

\makeatletter

\newcounter{quotethmcnt}

\ifx\MyBeamer\undefined

\def\equationautorefname~#1\null{(#1)}
\def\itemautorefname~#1\null{#1}
\fi

\ifx\MyBeamer\undefined

\ifx\thmnum\undefined
\newtheorem{thmnum}{}[section]
\fi

\newcommand{\mynewthm}[3][]{%
  \newaliascnt{#2}{thmnum}%
  \newtheorem{#2}[#2]{#3}%
  \aliascntresetthe{#2}%
  \newtheorem*{#2*}{#3}%
  \expandafter\newcommand\csname #2autorefname\endcsname{#3}%
  \expandafter\renewcommand\csname the#2\endcsname{\thethmnum}%
}

\newtheorem*{clm}{Claim}
\newenvironment{clmprf}{%
  \begin{proof}[Proof of claim]%
  }{\end{proof}}

\else

\let\xxx=\frametitle
\def\frametitle#1{%
  \xxx{%
    \setbeamercolor*{math text}{use={titlelike,my math text},fg=titlelike.fg!80!my math text.fg}%
    #1}%
  \setbeamercolor{math text}{use=my math text,fg=my math text.fg}%
}

\newcommand{\beamerenv}[3]{%
\newenvironment<>{#1}%
{%
  \setbeamercolor{temp}{fg=structure.fg}%
  \setbeamercolor{structure}{fg=#2}%
  \setbeamercolor{block body}{use=structure,bg=structure.fg!5!white}%
  \begin{#3}%
}%
{\end{#3}\setbeamercolor{structure}{fg=temp.fg}}}

\newcommand{\mynewthm}[3][green!50!black]{%
  \newtheorem*{#2x}{#3}%
  \beamerenv{#2}{#1}{#2x}%
}

\beamerenv{other}{violet}{block}

\fi

\ifx\FrenchText\undefined
\newcommand{\myiffrench}[2]{#2}
\else
\newcommand{\myiffrench}[2]{\iflanguage{french}{#1}{#2}}
\fi

\theoremstyle{plain}
\mynewthm[red]{thm}{\myiffrench{Théorème}{Theorem}}
\mynewthm{prp}{Proposition}
\mynewthm[purple]{lem}{\myiffrench{Lemme}{Lemma}}
\mynewthm[orange]{fct}{\myiffrench{Fait}{Fact}}
\mynewthm[purple!50!white]{cor}{\myiffrench{Corollaire}{Corollary}}

\theoremstyle{definition}
\mynewthm[green!80!black]{dfn}{\myiffrench{Définition}{Definition}}
\mynewthm{hyp}{\myiffrench{Hypothèse}{Hypothesis}}
\mynewthm{conv}{Convention}
\mynewthm{conj}{Conjecture}
\mynewthm{ntn}{Notation}
\mynewthm{cst}{Construction}

\theoremstyle{remark}
\mynewthm[yellow!90!black]{rmk}{\myiffrench{Remarque}{Remark}}
\mynewthm{qst}{Question}
\mynewthm{exc}{\myiffrench{Exercice}{Exercise}}
\mynewthm{exm}{\myiffrench{Exemple}{Example}}

\ifx\MyBeamer\undefined

\newcommand{\myenumlabel}[1]{\textnormal{(\roman{#1})}}
\renewcommand{\theenumi}{\myenumlabel{enumi}}
\fi

\newcounter{cycprfcnt}
\newcounter{cycprffirst}
\newcommand{\cycprfpreamble}[1]%
{%
  \setcounter{cycprfcnt}{1}
  \setcounter{cycprffirst}{#1}
  \setlength{\itemindent}{0.5\leftmargin}%
  \setlength{\leftmargin}{0pt}%
  \newcommand{\cpcurr}{\myenumlabel{cycprfcnt}}%
  \newcommand{\cpnext}{\addtocounter{cycprfcnt}{1}\cpcurr}%
  \newcommand{\cpprev}{\addtocounter{cycprfcnt}{-1}\cpcurr}%
  \newcommand{\cpnum}[1]{\setcounter{cycprfcnt}{##1}\cpcurr}%
  \newcommand{\cpfirst}{\cpnum{1}}%
  \newcommand{\impnext}{\cpcurr{} $\Longrightarrow$ \cpnext.}%
  \def\makelabel##1{\ifnum\value{cycprffirst}=0\hspace{-0.7\itemindent}\setcounter{cycprffirst}{1}\fi##1}%
}%

{\begin{list}{\impnext}{\cycprfpreamble{#1}}}%
{\qedhere\end{list}}%

\newenvironment{cycprf*}[1][0]%
{\begin{list}{\impnext}{\cycprfpreamble{#1}}}%
{\end{list}}%

\def\indsym#1#2{%
  \setbox0=\hbox{$\m@th#1x$}%
  \kern\wd0%
  \hbox to 0pt{\hss$\m@th#1\mid$\hbox to 0pt{$\m@th#1^{#2}$\hss}\hss}%
  \lower.9\ht0\hbox to 0pt{\hss$\m@th#1\smile$\hss}%
  \kern\wd0}

\def\nindsym#1#2{%
  \setbox0=\hbox{$\m@th#1x$}%
  \kern\wd0%
  \hbox to 0pt{\hss$\m@th#1\not$\kern1.4\wd0\hss}
  \hbox to 0pt{\hss$\m@th#1\mid$\hbox to 0pt{$\m@th#1^{#2}$\hss}\hss}%
  \lower.9\ht0\hbox to 0pt{\hss$\m@th#1\smile$\hss}%
  \kern\wd0}

\def\dotminussym#1#2{%
  \setbox0=\hbox{$\m@th#1-$}%
  \kern.5\wd0%
  \hbox to 0pt{\hss\hbox{$\m@th#1-$}\hss}%
  \raise.6\ht0\hbox to 0pt{\hss$\m@th#1.$\hss}%
  \kern.5\wd0}

\renewcommand{\emptyset}{\varnothing}
\renewcommand{\setminus}{\smallsetminus}

\usepackage{mathpazo}

\newcommand{\qf}{\mathrm{qf}}

\DeclareMathOperator{\tS}{S}

\DeclareMathOperator{\co}{co}

\newcommand{\cB}{\mathcal{B}}

\newcommand{\cE}{\mathcal{E}}
\newcommand{\cF}{\mathcal{F}}

\newcommand{\cU}{\mathcal{U}}
\newcommand{\cV}{\mathcal{V}}
\newcommand{\cW}{\mathcal{W}}

\newcommand{\bN}{\mathbf{N}}

\newcommand{\bS}{\mathbf{S}}

\newcommand{\bZ}{\mathbf{Z}}

\makeatother

\usepackage{stmaryrd}

\begin{document}

\title{Isometrisable group actions}

\author{Itaï \textsc{Ben Yaacov}}

\address{Itaï \textsc{Ben Yaacov} \\
  Université Claude Bernard -- Lyon 1 \\
  Institut Camille Jordan, CNRS UMR 5208 \\
  43 boulevard du 11 novembre 1918 \\
  69622 Villeurbanne Cedex \\
  France}

\urladdr{\url{http://math.univ-lyon1.fr/~begnac/}}

\author{Julien \textsc{Melleray}}

\address{Julien \textsc{Melleray} \\
  Université Claude Bernard -- Lyon 1 \\
  Institut Camille Jordan, CNRS UMR 5208 \\
  43 boulevard du 11 novembre 1918 \\
  69622 Villeurbanne Cedex \\
  France}

\urladdr{\url{http://math.univ-lyon1.fr/~melleray/}}

\thanks{Research supported by ANR project GruPoLoCo (ANR-11-JS01-008).}

\svnInfo $Id: Isomble.tex 2574 2015-10-29 16:24:02Z begnac $
\thanks{\textit{Revision} {\svnInfoRevision} \textit{of} \today}


\begin{abstract}
  Given a separable metrisable space $X$, and a group $G$ of homeomorphisms of $X$, we introduce a topological property of the action $G \curvearrowright X$ which is equivalent to the existence of a $G$-invariant compatible metric on $X$. This extends a result of Marjanović obtained under the additional assumption that $X$ is locally compact.
\end{abstract}

\maketitle

\tableofcontents

\section*{Introduction}

This paper grew out of the following question: given a metrisable topological space $X$, and a homeomorphism $g$ of $X$, how can one determine whether there exists a distance inducing the topology of $X$ and for which $g$ is an isometry?
More generally, it is interesting to determine when there exists a compatible invariant distance for an action by homeomorphisms of some group $G$ on $X$.
When this happens we say that the action $G \curvearrowright X$ is \emph{isometrisable}.

When $X$ is compact, this problem is well understood, and various characterisations are available -- for instance, in that case an action $G \curvearrowright X$ is isometrisable if and only if it is equicontinuous, in the sense that for any open $U \subseteq X \times X$ containing the diagonal $\Delta_X$, there exists an open $V \subseteq X \times X$ containing $\Delta_X$ and such that for all $g \in G$ one has $(g \times g) V \subseteq U$. One way to prove this is to note that, if the latter property holds, then the sets $G \cdot V$ form a countably generated uniformity which is compatible with the topology and admits a basis of invariant entourages, and such a uniformity comes from a $G$-invariant metric (as a general reference about uniformities and the basic facts about them used in this paper, the reader may consult Chapter 8 in \cite{Engelking:GeneralTopology}).
One could equivalently formalise the previous condition by saying that $G$ generates a relatively compact subgroup of the group of homeomorphisms of $X$, endowed with the topology of uniform convergence for some compatible distance on $X$ (here the nontrivial direction follows from the Arzel\`a-Ascoli theorem, or by considering averages of any compatible metric against the Haar measure of $G$).

Beyond that, only the locally compact case seems to be addressed in the literature.
Marjanović \cite{Marjanovic:TopologicalIsometries} appears to be the first with a significant result in this direction.
In order to formulate it, we recall that, if $\cF$ is a family of continuous maps from a topological space $X$ to a topological space $Y$, $\cF$ is said to be \emph{evenly continuous} if for all $x \in X$, all $y \in Y$ and all open $V \ni y$, there exists an open $U \ni x$ and an open $W$ with $y \in W \subseteq V$ and such that
\[
\forall f \in \cF \ f(x) \in W \Rightarrow f(U) \subseteq V.
\]

\begin{thm*}[Marjanović \cite{Marjanovic:TopologicalIsometries}]
  Let $X$ be a locally compact separable metrisable space, and $f$ be homeomorphism of $X$. Then there is a compatible distance for which $f$ is an isometry if, and only if, the family
  $\{f^n : n \in \bZ\}$ is evenly continuous from $X$ to its Alexandrov compactification.
\end{thm*}

This result was slightly extended by Borges \cite{Borges:Recognize} and Kiang \cite{Kiang:SemigroupsOfMappings}; it follows from Kiang's work that Marjanović's result extends to arbitrary groups acting on locally compact separable metrisable spaces (though that fact is not explicitly formulated in \cite{Kiang:SemigroupsOfMappings} and could also be deduced from Marjanović's argument, it is a direct consequence of the main theorem of \cite{Kiang:SemigroupsOfMappings}).
One obstacle to extend these results beyond the locally compact case is the presence of the Alexandrov compactification in the statement; another is that Marjanović's and Kiang's arguments rely heavily on compactness.
To address the first issue, one might try considering a stronger property than even continuity.

\begin{dfn*}[Royden \cite{Royden:RealAnalysis}]
  The action $G \curvearrowright X$ is \emph{topologically equicontinuous} if, for any $x,y \in X$ and any open subset $V \ni y$, there exists open subsets $W \ni x$ and $y \in U \subseteq V$ such that
  $$\forall g \in G \quad \left( g W \cap U \ne \emptyset \right) \Rightarrow gW \subseteq V \ . $$
\end{dfn*}

It is obvious that, if $G \curvearrowright X$ is isometrisable, then it is topologically equicontinuous. It is also not hard to check that, when $X$ is locally compact, even continuity of $G$ as a family of maps from $X$ to its Alexandrov compactification is equivalent to topological equicontinuity of $G$ as a family of maps from $X$ to itself.
Topological equicontinuity is a strong assumption, and we discuss some consequences in the second section. It appears not to be sufficient for isometrisability of the action $G \curvearrowright X$, leading us to consider an even stronger property.

\begin{dfn*}
  We say that $G \curvearrowright X$ is \emph{uniformly topologically equicontinuous} if, for any $y \in X$ and any open $V \ni y$, there exists an open $U$ with $y \in U \subseteq V$ and such that for all $x \in X$ there exists an open neighborhood $W$ of $x$ satisfying
  \[
  \forall g \in G \ (gW \cap U \ne \emptyset) \Rightarrow gW \subseteq V \ .
  \]
\end{dfn*}

Our main result is the following.

\begin{thm*}
  Let $X$ be a second countable metrisable space, and $G$ be a group acting on $X$ by homeomorphisms. Then the action $G \curvearrowright X$ is isometrisable if, and only if, it is uniformly topologically equicontinuous.
\end{thm*}

Rahter than directly defining a $G$-invariant compatible metric under the assumption of uniform topological equicontinuity, our argument proceeds by building a countably generated uniformity with a basis of $G$-invariant entourages, then using the fact that such an uniformity is generated by $G$-invariant pseudometrics, and finally using second countability to subsume this family into one $G$-invariant metric.

\section{Topological equicontinuity}

Throughout the text $X$ stands for a separable metrisable space, and $G$ is a group of homeomorphisms of $X$.

\begin{lem}
  \label{lem:Equicontinuity}
  Assume that $G \curvearrowright X$ is topologically equicontinuous.
  Assume that $x_n \rightarrow x$ and $y_n \rightarrow y$ in $X$, and let $(g_n) \in G^\bN$ be such $g_n x_n \rightarrow y$.
  Then $g_n^{-1} y_n \rightarrow x$.

  In particular, $g_n x \rightarrow y$ if and only if $g_n^{-1} y \rightarrow x$; thus a topologically equicontinuous action is minimal if and only if it is topologically transitive.
\end{lem}
\begin{proof}
  Fix $U$ open containing $x$, and find $V$ open contained in $U$ and containing $x$, $W$ open containing $y$ such that $(gW \cap V \ne \emptyset) \Rightarrow gW \subseteq U$ for any $g \in G$.
  For $n$ large enough $x_n \in V$ and $y_n, g_n x_n \in W$, so $g_n^{-1} W \subseteq U$, and in particular $g_n^{-1} y_n \in U$, as desired.

  To see why topological transitivity implies minimality, assume the action is topologically transitive (that is, for any nonempty open $U,V$ there exists $g \in G$ such that $gU \cap V \ne \emptyset$) and pick $x,y \in X$.
  By assumption, there exist $g_n \in G$ and $x_n \in X$ such that $x_n$ converges to $x$ and $g_n x_n $ converges to $y$. Hence $g_n^{-1}y$ converges to $x$, showing that the orbit of $y$ is dense.
\end{proof}

\begin{prp}
  \label{prp:Minimal}
  Assume that $G \curvearrowright X$ is minimal.
  Then $G \curvearrowright X$ is isometrisable if, and only if, it is topologically equicontinuous.
\end{prp}
\begin{proof}
  One implication is clear.
  For the other, assume that $G \curvearrowright X$ is topologically equicontinuous, and denote by $\tau$ the topology of $X$.
  Consider the family of sets of the form $G \cdot U^2 \subseteq X^2$ where $U$ varies over all nonempty open sets in $X$.

  Since the action is minimal, $G \cdot U^2$ contains the diagonal.
  Given such a set $G \cdot U^2$, find an open $\emptyset \neq V \subseteq U$ such that $gV \cap V \ne \emptyset \Rightarrow gV \subseteq U$.
  Assume now that $(x,y),(y,z) \in G \cdot V^2$, say $x,y \in h_1V$ and $y,z \in h_2V$.
  Then $h_1^{-1} y \in V \cap h_1^{-1}h_2 V$, hence $h_1^{-1}h_2 V \subseteq U$, so $h_1^{-1} z \in U$. Thus both $h_1^{-1}x$ and $h_1^{-1}z$ belong to $U$, and $(x,z) \in G \cdot U^2$.

  Thus the sets $G \cdot U^2$ form a basis of entourages for a uniformity, which is metrisable by a $G$-invariant distance $d$ since it is countably generated by $G$-invariant entourages, and we claim that it is compatible with the topology on $X$.
  Since each $G \cdot U^2$ is open, $d$ is continuous.
  Conversely, assume that $x_n \rightarrow^d x$.
  For every neighbourhood $U$ of $x$ we have $(x_n,x) \in G \cdot U^2$ for all $n$ large enough, giving rise to a sequence $(g_n)$ such that $g_n x_n \rightarrow x$ and $g_n x \rightarrow x$.
  By \autoref{lem:Equicontinuity} we have $x_n \rightarrow x$.
  Therefore $d$ is a compatible metric.
\end{proof}

When the action $G \curvearrowright X$ is assumed to be transitive, the above result appears as an exercise in Royden \cite{Royden:RealAnalysis}.

Next we introduce the ``topological ergodic decomposition'' associated to $G \curvearrowright X$.

\begin{dfn}
  \label{dfn:Equivalence}
  For $x,y \in X$, let $[x] = \overline{G x}$ and say that $x \sim y$ if $x \in [y]$.
\end{dfn}

\begin{lem} \label{lem:Equivalence}
  Assume that $G \curvearrowright X$ is topologically equicontinuous.
  Then the relation $\sim$ is a closed equivalence relation on $X$ (i.e., it is an equivalence relation which is closed as a subset of $X \times X$).

  In particular, $[x] = \{y : x \sim y\}$ and $x \sim y$ if and only if $[x] \cap [y] \neq \emptyset$.
\end{lem}
\begin{proof}
  It is clear that $\sim$ is transitive and reflexive, and when $G \curvearrowright X$ is topologically equicontinuous it is also symmetric, by \autoref{lem:Equicontinuity}.
  In order to see that it is closed, assume that $(x_n,y_n) \rightarrow (x,y)$ in $X^2$, where $x_n \sim y_n$ for all $n$.
  If $U \ni x$ is open, for $n$ large enough we have $x_n \in U$, and since $x_n \sim y_n$, where $g_n \in G$ such that $g_n y_n \in U$.
  We thus construct a sequence $(g_n)$ with $g_n y_n \rightarrow x$.
  By \autoref{lem:Equicontinuity} we have $g_n^{-1} x \rightarrow y$ and $y \sim x$, as desired.
\end{proof}

Consequently the quotient space which we denote by $X \sslash G$ is Hausdorff.
If the action $G \curvearrowright X$ is isometrisable then this quotient must be metrisable.

\begin{lem}
  \label{lem:OpenQuotient}
  Assume that $G \curvearrowright X$ is topologically equicontinuous. Then the projection map $\pi \colon X \to X \sslash G$ is open.
\end{lem}
\begin{proof}
  Let $U \subseteq X$ be open, $x \in U$ and $y \sim x$.
  Then $G y \cap U \neq \emptyset$, or equivalently, $y \in GU$.
  It follows that the open set $GU$ is the $\sim$-saturation of $U$, so $\pi U$ is open.
\end{proof}


\section{Uniform topological equicontinuity and isometrisability}

Metrisability of $X \sslash G$ is obviously a necessary condition for the action $G \curvearrowright X$ to be isometrisable.
Outside the realm of locally compact spaces, this seems to require a stronger hypothesis than mere topological equicontinuity.

\begin{dfn}
  We say that $G \curvearrowright X$ is \emph{uniformly topologically equicontinuous} if for any $x \in X$ and any open $V \ni x$ there exists an open $U$ with $x \in U \subseteq V$ such that for all $y \in X$ there exists an open $W_y \ni y$ satisfying
  \[
  \forall g \in G \ (gW_y \cap U \ne \emptyset) \Rightarrow gW_y \subseteq V.
  \]
  When the conditions above are satisfied, we say that $U$ witnesses uniform topological equicontinuity for $x,V$.
\end{dfn}

This definition is obtained by inverting two quantifiers in the definition of topological equicontinuity, and is still a necessary condition for isometrisability of $G \curvearrowright X$.

\begin{prp}
  Assume that $G \curvearrowright X$ is uniformly topologically equicontinuous.
  Then $X \sslash G$ is metrisable.
\end{prp}
\begin{proof}
  Since $X$ is second countable so is $X \sslash G$, and it will suffice to prove that $X \sslash G$ is regular.
  In other words, we need to prove that given a closed $G$-invariant $F \subseteq X$ and $x \not \in F$, there exist open sets $U \ni x$ and $W \supseteq F$ such that $U \cap GW = \emptyset$.
  We choose $U$ which witnesses uniform topological equicontinuity for $x, X \setminus F$, and for each $y \in F$ we let $W_y \ni y$ be the corresponding neighbourhood.
  If there existed $y \in F$ and $g \in G$ such that $g W_y \cap U \neq \emptyset$ then $g W_y \subseteq X \setminus F$ and in particular $gy \notin F$, a contradiction.
  Therefore $U \cap \bigcup_{y \in F} GW_y = \emptyset$, which is enough.
\end{proof}

Given Marjanović's result recalled in the introduction, the following fact is worth mentioning.
(If one merely wishes to prove that $X \sslash G$ is metrisable when $X$ is locally compact and the action is topologically equicontinuous, a much shorter argument exists.)

\begin{prp}
  Let $X$ be a locally compact separable metrisable space, and $G$ a group acting on $X$ by homeomorphisms.
  The following conditions are equivalent.
  \begin{enumerate}
  \item \label{p:topeq1} $G \curvearrowright X$ is uniformly topologically equicontinuous.
  \item \label{p:topeq2} $G \curvearrowright X$ is topologically equicontinuous.
  \item \label{p:topeq3} $G$, seen as a family of maps from $X$ to its Alexandrov compactification $X^*$, is evenly continuous.
  \end{enumerate}
\end{prp}
\begin{proof}
  Note that \autoref{p:topeq3} is equivalent to saying that, for all $x \in X$ and $y \in X^*$, if $(x_i)$ converges to $x$ and $(g_i x) $ converges to $y$ then $(g_i x_i)$ also converges to $y$.

  The implication \autoref{p:topeq1} $\Rightarrow$ \autoref{p:topeq2} is by definition. To see that $\autoref{p:topeq2}$ implies $\autoref{p:topeq3}$, assume that there exists $x \in X$ and a compact $K \subseteq X$ such that for all open $U \ni x$ and for all compact $L \supseteq K$ there is $g \in G$ such that $g(x) \not \in L$ and $g(U) \cap K \ne \emptyset$.
  From this we may build a sequence $(x_i)$ converging to $x$ and elements $g_i \in G$ such that $g_i(x) \to \infty$ and $g_i(x_i) \to k \in K$. This is incompatible with \autoref{p:topeq2}.

  It remains to prove that \autoref{p:topeq3} $\Rightarrow$ \autoref{p:topeq1}.
  We again proceed by contradiction and assume that $G \curvearrowright X$ is not uniformly topologically equicontinuous but $G$ is an evenly continuous family of maps from $X$ to $X^*$.
  By assumption, there exists $y \in X$ and an open $V \ni y$ such that for any open $U$ with $y \in U \subseteq V$ there exists $x \in X$ such that for all open $W \ni x$ there exists $g \in G$ with both $gW \cap U \ne \emptyset$ and $gW \not \subseteq V$.
  Letting $U$ vary over a basis of open neighborhoods of $y$, we obtain a sequence $(x_i)$ witnessing the above condition; up to extractions we see that there are two cases to consider:

  \begin{itemize}
  \item $(x_i)$ converges to some $x \in X$.
    Then there exists sequences $(g_i)$ and $(y_i), (z_i)$ converging to $x$ such that $g_i y_i$ converges to $y$ and $g_i z_i$ lives outside $V$.
    Up to some extraction, we may assume that $g_i x$ and $g_i z_i$ both converge in $X^*$, and the fact that $g_i y_i$ and $g_i z_i$ have different limits shows that even continuity must be violated at $x$.

  \item $(x_i)$ converges to $\infty$, and for all compact $K$ there exists $I$ such that for all $i \ge I$ and all $g$ one has $gx_i \not \in K$ (otherwise, replacing $x_i$ by some $g_i x_i$ and going to a subsequence we would be in the situation of the first case above).
    Letting $U$ be a relatively compact neighborhood of $y$, we see that for $i$ large enough we have $ \overline{G x_i} \cap \overline{ U} = \emptyset$.
    Then the even continuity of $G$ implies that there must exist some neighborhood $W$ of $x_i$ such that $G W \cap \overline{U} = \emptyset$ (by essentially the same argument as above), which contradicts the choice of $x_i$.
  \end{itemize}
\end{proof}

\begin{dfn}
  Let $\cU$ be an open cover of $X$.
  We say that it is \emph{$G$-invariant} if for any $U$ one has $U \in \cU \Rightarrow gU \in \cU$.

  A \emph{$G$-basis} of a $G$-invariant open cover $\cU$ is a subset $\cB$ such that all elements of $\cU$ are of the form $gB$ for some $B \in \cB$.

  We say that a $G$-invariant open cover $\cU$ is \emph{$G$-locally finite} if it admits a $G$-basis $\cB$ such that for any $x \in X$ there exists a neighborhood $A$ of $x$ (not necessarily belonging to $\cU$) such that $\{B \in \cB : \exists g \in G \ gB \cap A \ne \emptyset\}$ is finite.
\end{dfn}

\begin{lem} \label{l:locfin}
  Assume that $G \curvearrowright X$ is uniformly topologically equicontinuous. Then any $G$-invariant open cover admits a $G$-locally finite open refinement.
\end{lem}
\begin{proof}
  Let $\cU$ be a $G$-invariant open cover.
  Let also $\pi\colon X \rightarrow X \sslash G$ denote the open quotient map.
  Since $X \sslash G$ is metrisable, it is paracompact (see e.g.\ Theorem 4.4.1 in \cite{Engelking:GeneralTopology}), so we can find a locally finite refinement $\cV$ of $\pi \cU$.
  For any $V \in \cV$, pick some $U_V \in \cU$ such that $\pi(U_V) \supseteq V$, and set $W_V=U_V \cap \pi^{-1}(V)$.
  Let $\cW$ be the $G$-invariant open cover with $G$-basis $\{W_V : V \in \cV\}$.
  By construction $\cW$ is an open cover and refines $\cU$.

  Now pick any $x \in X$.
  There is an open neighborhood $O$ of $\pi(x)$ which meets only finitely many elements of $\cV$.
  If $W_V$ is such that $g W_V \cap \pi^{-1}(O) \ne \emptyset$ for some $g \in G$, then $V \cap O \ne \emptyset$, so $\cW$ is $G$-locally finite.
\end{proof}

\begin{ntn}
  To an open cover $\cU$ of $X$ we associate an entourage $\cE(\cU) = \bigcup_{U \in \cU} U^2 \subseteq X^2$.
\end{ntn}

\begin{lem}\label{l:triangle}
  Assume that $G \curvearrowright X$ is uniformly topologically equicontinuous.
  Let $\cU$ be a $G$-invariant open cover of $X$.
  Then there exists a $G$-invariant open refinement $\cV$ of $\cU$ with the property that for all $x,y,z \in X$, if $(x,y),(y,z) \in \cE(\cV)$ then $(x,z) \in \cE(\cU)$.
\end{lem}
\begin{proof}
  Uniform topological equicontinuity of the action enables us to find a $G$-invariant open refinement $\cW$ of $\cU$ with the property that for all $W \in \cW$ there exists $U(W) \in \cU$ containing $W$ such that for all $y \in X$ there exists an open $C_{W,y} \ni y$ satisfying $gC_{W,y} \cap W \ne \emptyset \Rightarrow gC_{W,y} \subseteq U(W)$.
  Using \autoref{l:locfin} we may assume that $\cW$ is $G$-locally finite and $\cB$ is a $G$-basis of $\cW$ witnessing that property.
  We let $\cV$ consist of all open sets $V$ such that for all $g \in G$ and $B \in \cB$:
  \begin{gather*}
    gV \cap B \neq \emptyset \Rightarrow gV \subseteq U(B).
  \end{gather*}
  Given $x \in X$ there exists an open $A \ni x$ such that $\cB_A=\{B \in \cB : gB \cap A \ne \emptyset\}$ is finite, so $x \in A \cap \bigcap_{B \in \cB_A} C_{B,x} \in \cV$.
  Thus $\cV$ is a cover, and it is clearly $G$-invariant and refines $\cU$.

  Assume now that $(x,y),(y,z) \in \cE(\cV)$, say $x,y \in V_1$ and $y,z \in V_2$ where $V_i \in \cV$.
  There exist some $B \in \cB$ and $g \in G$ such that $gy \in B$, so $gV_1 \cup gV_2 \subseteq U(B)$ and $x,z \in g^{-1} U(B) \in \cU$.
\end{proof}

\begin{lem}\label{l:separation}
  Assume that $G \curvearrowright X$ is uniformly topologically equicontinuous, and fix $x \in X$.
  For any $G$-invariant open cover $\cU$ there exists a $G$-invariant open refinement of $\cU$ with $G$-basis $\cB$ and $B \in \cB$ such that for any $A \in \cB$ different from $B$ and any $g \in G$ one has $x \not \in gA$.
\end{lem}
\begin{proof}
  Pick $U \in \cU$ such that $x \in U$. Using uniform topological equicontinuity, choose an open neighborhood $V$ of $x$ such that for any $y \in X \setminus GU$ there exists an open set $W_y$ satisfying $g W_y \cap V = \emptyset$ for all $g \in G$. Refining if necessary, we may assume that each $W_y$ is contained in some element of $\cU$; then $\{W_y : y \in X \setminus GU\} \cup \{U\}$ form a $G$-basis for a $G$-invariant open refinement of $\cU$ with the desired property.
\end{proof}

\begin{lem}\label{l:pseudometric}
  Assume that $G \curvearrowright X$ is uniformly topologically equicontinuous.
  Then for any $x \in X$ there exists a continuous $G$-invariant pseudometric $d_x$ such that $d_x(x_i,x)$ converges to $0$ if and only if $(x_i)$ converges to $x$.
\end{lem}
\begin{proof}
  Fix $x \in X$.
  Using lemmas \ref{l:separation} and \ref{l:triangle}, we can build a sequence of $G$-invariant coverings $\cU_n$ of $X$ with $G$-basis $\cB_n$ with the following properties:
  \begin{itemize}
  \item For each $n$ there exists a unique $B_n \in \cB_n$ such that $x \in GB_n$, and $\{B_n\}_n$ forms a basis of neighbourhoods of $x$.
  \item For all $n$, if $(y,z),(z,t) \in \cE_{n+1}$ then $(y,t) \in \cE_n$, where $\cE_n = \cE(\cU_n)$.
  \end{itemize}
  As in the proof of \autoref{prp:Minimal}, the family of entourages $\cE_n$ gives rise to a uniformity which is moreover metrisable by $G$-invariant pseudo-metric $d_x$.
  Since all entourages are open, $d_x$ is continuous.

  Assuming that $d_x(x_i,x) \rightarrow 0$, for all $n$ there must exist some $U_n \in \cU_n$ such that $x_i,x \in U_n$ for all $i$ large enough, and $U_n$ must be of the form $g_nB_n$.
  It follows that there exists a sequence $h_i$ with $h_i x_i \rightarrow x$ and $h_i x \rightarrow x$, so $x_i \rightarrow x$ by \autoref{lem:Equicontinuity}.
\end{proof}

\begin{thm}
  \label{thm:Isometrisable}
  Let $X$ be a second-countable metrisable space, and let $G$ act on $X$ by homeomorphisms.
  Then $G \curvearrowright X$ is isometrisable if and only if it is uniformly topologically equicontinuous.
\end{thm}
\begin{proof}
  One direction is clear, so we prove the other.
  Applying \autoref{l:pseudometric}, we obtain a family of continuous $G$-invariant pseudometrics $(d_x)_{x \in X}$.
  Since $d_x(x,x_i) \rightarrow 0$ if and only if $x_i \rightarrow x$, for any open subset $U$ of $X$ and any $x \in U$, there exists $\varepsilon >0$ such that $d_x(x,y) < \varepsilon \Rightarrow y \in U$. By the Lindel\"off property, we obtain that for any open $U \subseteq X$ there exists a countable subset $A \subseteq U$ and a family $(\varepsilon_a)_{a \in A}$ such that
  \[
  U= \bigcup_{a \in A} \{x : d_a (x,a) < \varepsilon_a \}.
  \]
  Applying this to a countable basis for the topology of $X$, we obtain a countable family of $G$-invariant pseudometrics which generate the topology, and this countable family may be subsumed into a single $G$-invariant metric.
\end{proof}

\section{Complete and incomplete metrics}

In this section we assume that $G \curvearrowright X$ is isometrisable, and $X$ admits a compatible complete metric. A natural question is then: must $G \curvearrowright X$ admit a complete invariant metric?

The following observation is immediate.

\begin{prp}
  \label{prp:MinimalComplete}
  Assume that $G \curvearrowright X$ is minimal.
  Then there exists a complete compatible $G$-invariant distance on $X$ if and only if all compatible $G$-invariant distances are complete.
\end{prp}
\begin{proof}
  Assume that a compatible complete $G$-invariant distance exists, fix $x \in X$, and let $d$ be another $G$-invariant compatible distance.
  Then a sequence $(g_i x)$ is $d$-Cauchy if and only if for any neighborhood $V$ of $x$ there exists $N$ such that $g_i^{-1} g_j x \in V$ for all $i, j \ge N$. This property does not depend on $d$ but only on the topology of $X$, so any $d$-Cauchy sequence of the form $(g_i x)$ must converge.

  Given any $d$-Cauchy sequence $(x_i)$, the minimality of $G \curvearrowright X$ enables us to find $g_i$ such that $d(g_i x, x_i) < 2^{-i}$ for all $i$. Then $(g_i x)$ is also $d$-Cauchy, hence convergent, and so is $(x_i)$.
\end{proof}
The fact above is well-known in the particular case when $G$ is a Polish group acting by left translation on itself (and the proof is the same). All Polish groups admit left-invariant compatible metrics, but not all of them admit such metrics which are also complete (and, if one such metric is complete, all of them are). For instance, the group $S_\infty$ of all permutations of the integers, endowed with its usual Polish topology, does not admit a compatible left-invariant complete metric.

The following simple example was suggested by C. Rosendal.
\begin{exm}
  There exists a Polish space $X$ and a $\bZ$-action on $X$ which is isometrisable but which admits no complete invariant distance.
\end{exm}
\begin{proof}
  Let $r$ be an irrational rotation of the unit circle $\bS$, and let $X = \bS \setminus \{r^i(1) : i \in \bZ\}$. Then $X$ is a $G_\delta$ subset of $\bS$, hence Polish, and the restriction of $r$ to $X$ generates an isometrisable $\bZ$-action; the metric on $X$ induced from the usual metric on $\bS$ is both invariant and not complete, so there cannot exist an invariant complete metric on $X$.
\end{proof}

As it turns out, the minimal case contains essentially all the obstructions to the existence of a complete invariant metric.

\begin{thm}
  Assume that $X$ is completely metrisable and the action $G \curvearrowright X$ is isometrisable.
  Then there exists a compatible complete $G$-invariant distance on $X$ if and only if there exists such a distance on the closure of each $G$-orbit.
\end{thm}
\begin{proof}
  The condition is clearly necessary.
  Now assume that there exists a compatible complete metric on the closure of each $G$-orbit, and let $d$ be a $G$-invariant metric on $X$.
  By \autoref{prp:MinimalComplete} the restriction of $d$ to each $[x] = \overline{G x}$ is complete.
  Also, since the projection map $X \to X \sslash G$ is open and open maps with range in a metrisable space preserve complete metrisability (see for instance Exercise 5.5.8(d) p.341 in \cite{Engelking:GeneralTopology}), there exists a complete distance $\rho$ on $X \sslash G$.
  Consider the new metric $d'$ defined by
  $$ d'(x,y)= d(x,y) + \rho([x],[y]). $$
  Clearly $d'$ is $G$-invariant and compatible with the topology of $X$.
  Assume now that $(x_n)$ is $d'$-Cauchy.
  Since $\rho$ is complete, $[x_n]$ must converge to $[x]$ for some $x \in X$, i.e., there exists a sequence $(g_n)$ such that $g_n x_n \rightarrow x$.
  By invariance, $d(x_n,g_n^{-1}x) \to 0$, so $(g_n^{-1}x)$ is a $d$-Cauchy sequence in $[x]$ which must converge to some $y$.
  Therefore $x_n \rightarrow y$ as well, concluding the proof.
\end{proof}

\providecommand{\bysame}{\leavevmode\hbox to3em{\hrulefill}\thinspace}


\begin{thebibliography}{Mar69}

\bibitem[Bor71]{Borges:Recognize}
Carlos~R. \bgroup\scshape{}Borges\egroup{}, \emph{How to recognize
  homeomorphisms and isometries}, Pacific Journal of Mathematics \textbf{37}
  (1971), 625--633.

\bibitem[Eng89]{Engelking:GeneralTopology}
Ryszard \bgroup\scshape{}Engelking\egroup{}, \emph{General topology}, second
  ed., Sigma Series in Pure Mathematics, vol.~6, Heldermann Verlag, Berlin,
  1989, Translated from the Polish by the author.

\bibitem[Kia73]{Kiang:SemigroupsOfMappings}
Mo~Tak \bgroup\scshape{}Kiang\egroup{}, \emph{On some semigroups of mappings},
  Nederl. Akad. Wetensch. Proc. Ser. A {\bf 76}=Indag. Math. \textbf{35}
  (1973), 18--22.

\bibitem[Mar69]{Marjanovic:TopologicalIsometries}
Milosav~M. \bgroup\scshape{}Marjanovi{\'c}\egroup{}, \emph{On topological
  isometries}, Nederl. Akad. Wetnesch. Proc. Ser. A 72=Indag. Math. \textbf{31}
  (1969), 184--189.

\bibitem[Roy88]{Royden:RealAnalysis}
H.~L. \bgroup\scshape{}Royden\egroup{}, \emph{Real analysis}, third ed.,
  Macmillan Publishing Company, New York, 1988.

\end{thebibliography}
\end{document}